\documentclass[11pt,english]{smfart}
\usepackage[T1]{fontenc}
\usepackage[utf8]{inputenc}

\usepackage[english]{babel}
\usepackage{smfthm,amssymb}

\usepackage{graphicx}

\input{macros-perso-fm.sty}

\theoremstyle{plain}%
\newtheorem{thm}{Theorem} 
\NumberTheoremsAs{thm}%

\newtheorem{cor}[thm]{Corollary}

\theoremstyle{definition}
\newtheorem{dfn}[thm]{Definition}

\theoremstyle{remark}

\newtheorem{ex}[thm]{Example}

\newtheorem{rem}[thm]{Remark}          
\newtheorem*{ack}{Acknowledgments}       
 \newtheorem*{conv}{Convention}

\author{Frédéric Mangolte}
\address{LUNAM Universit\'e, LAREMA, Universit\'e d'Angers} \email{frederic.mangolte@univ-angers.fr}
\urladdr{http://www.math.univ-angers.fr/~mangolte}
\title[Real rational surfaces]{Real rational surfaces} \alttitle{Real rational surfaces}
\date{\today}

\begin{document}

\frontmatter
%
\maketitle


\mainmatter
\section{Introduction}

During the last decade\footnote{With the exception of some classical references, only references over the past years from the preceding "RAAG conference in Rennes", which took place in 2001, are included.}, there were many progresses in the understanding of the topology of real algebraic manifolds, above all in dimensions $2$ and~$3$. Results on real algebraic threefolds were addressed in the survey \cite{ma-gaz} with a particular emphasis on Koll\'ar's results and conjectures concerning real uniruled and real rationally connected threefolds, see \cite{Ko01real}, \cite{hm1,hm2}, \cite{cm1,cm2}, \cite{mw1}. In the present paper, we will focus on real rational surfaces and especially on  their birational geometry.
Thus the three next sections are devoted to real rational surfaces; they are presented in a most elementary way. We state Commessatti's and Nash-Tognoli's famous theorems (Theorem~\ref{thm.com} and Theorem~\ref{thm.nash-to}). Among other things, we give a sketch of proof of the following statements:

\begin{itemize}
\item  Up to isomorphism, there is exactly one single real rational model of each nonorientable surface (Theorem~\ref{thm.bh}); 
\item The group of birational diffeomorphisms of a real rational surface is infinitely transitive (Theorem~\ref{thm.hm}); 
\item The group of birational diffeomorphisms of a real rational surface $X$ is dense in the group of $\cC^\infty$-diffeomorphisms $\Diff(\xr)$ (Theorem~\ref{thm.km1}). 
\end{itemize}

We conclude the paper with Section 5 devoted to a new line of research: the theory of \emph{regulous functions} and the geometry we are able to define with them.

Besides the progresses in the theory of real rational surfaces, the classification of other real algebraic surfaces has considerably advanced during the last decade (see \cite{Kh06} for a survey):
topological types and deformation types of real Enriques surfaces \cite{DIK}, deformation types of geometrically\footnote{See p.~\pageref{geomrat} before Theorem~\ref{thm.geomrat}.} rational surfaces \cite{DK02}, deformation types of real ruled surfaces \cite{We03}, topological types and deformation types of real bielliptic surfaces \cite{CF03}, topological types and deformation types of real elliptic surfaces \cite{am,bm1,DIK08}.

The present survey is an expansion of the preprint written by Johannes Huisman \cite{Hu11} from which we have borrowed several parts.

\begin{conv} 
In this paper, a \emph{real algebraic surface} (resp. \emph{real algebraic curve}) is a projective complex algebraic \emph{manifold} of complex dimension $2$ (resp.~$1$) endowed with an anti-holomophic involution whose set of fixed points is called \emph{the real locus} and denoted by $\xr$. A \emph{real map} is a complex map commuting with the involutions.
A \emph{topological surface} is a real $2$-dimensional $\cC^\infty$-manifold. By our convention, a real algebraic surface $X$ is nonsingular; as a consequence, if nonempty, the real locus $\xr$ gets a natural structure of a topological surface when endowed with the euclidean topology. Furthermore $\xr$ is compact since $X$ is projective.
\end{conv}

\begin{ack}
Thanks to Daniel Naie for sharing his picture of the real locus of a blow-up, see Figure~1, to Jérémy Blanc for old  references and the referee for useful remarks.
\end{ack}

\section{Real rational surfaces}
\subsection{Examples of rational surfaces}
\label{sec.real.rat}%
A real algebraic surface $X$ is \emph{rational} if it contains a Zariski-dense subset real isomorphic to the affine plane $\aA^2$. This is equivalent, as we shall see below,  to the fact that the function field of $X$ is isomorphic to the field of rational functions $\RR(x,y)$.  
In the sequel, a rational real algebraic surface will be called a \emph{real rational surface} for short and by our general convention, always assumed to be projective and nonsingular.

\begin{ex}
\label{ex.first}%
\begin{enumerate}
\item The real projective plane $\PP^2_{x:y:z}$ is rational. Indeed, each of the coordinate charts $U_0=\{x\ne 0\}$, $U_1=\{y\ne 0\}$, $U_2=\{z\ne 0\}$ is isomorphic to $\aA^2$. The real locus $\PP^2(\RR)$ endowed with the euclidean topology is the topological real projective plane.
\item The product surface $\PP^1_{x:y}\times\PP^1_{u:v}$ is rational. Indeed, the product open subset $\{x\ne 0\}\times\{u\ne 0\}$ is isomorphic to $\aA^2$. 
The set of real points $(\PP^1\times\PP^1)(\RR)=\PP^1(\RR)\times\PP^1(\RR)$ is diffeomorphic to the $2$-dimensional torus $\sS^1\times \sS^1$ where $\sS^1$ denotes the unit circle in $\RR^2$.
\item The quadric $Q_{3,1}$ in the projective space $\PP^3_{w:x:y:z}$ given by the affine equation $x^2+y^2+z^2=1$ is rational. 
Indeed, for a real point $P$ of $Q_{3,1}$, denote by $T_PQ_{3,1}$ the real projective plane in $\PP^3$ tangent to $Q_{3,1}$ at $P$. 
Then the stereographic projection $Q_{3,1}\setminus T_{P}Q_{3,1}\to \aA^2$ is an isomorphism of real algebraic surfaces. 
For example in the case $P$ is the North pole $N=[1:0:0:1]$, let 
$\pi_N\colon Q_{3,1}\to \PP^2_{U:V:W}$ be the rational map given by 
$$
\pi_N\colon [w:x:y:z]\dasharrow [x:y:w-z]\;.
$$ 
Then $\pi_N$ restricts to the stereographic projection from $Q_{3,1}\setminus T_{N}Q_{3,1}$ onto its image $\pi_N(Q_{3,1}\setminus T_{N}Q_{3,1})=\{w\ne 0\}\simeq \aA^2$.

(The inverse rational map $\pi_N^{-1}\colon \PP^2\dasharrow Q_{3,1}$ is given by 
$$
\pi_N^{-1}\colon [x:y:z]\dasharrow [x^2+y^2+z^2:2xz:2yz:x^2+y^2-z^2]) \;.
$$

The real locus $Q_{3,1}(\RR)$ is the unit sphere $\sS^2$ in $\RR^3$.
\end{enumerate}
\end{ex}

To produce more examples, we recall the construction of the blow-up which is especially simple in the context of rational surfaces.

The blow-up $B_{(0,0)}\aA^2$ of $\aA^2$ at $(0,0)$ is the quadric hypersurface defined in $\aA^2\times \PP^1$ by
$$
B_{(0,0)}\aA^2=
\{((x,y),[u:v])\in \aA^2_{x,y}\times \PP^1_{u:v} \st uy=vx\}.
$$

The blow-up $B_{[0:0:1]}\PP^2$ 
of $\PP^2$ at $P=[0:0:1]$ is the algebraic surface
\begin{multline*}
B_{[0:0:1]}\PP^2=
\{([x:y:z],[u:v])\in \PP^2_{x:y:z}\times \PP^1_{u:v} \st  uy-vx=0\}.
\end{multline*}

The open subset $V_0=\{((x,y),[u:v])\in B_{(0,0)}\aA^2\st u\ne 0\}$ is Zariski-dense in $B_{(0,0)}\aA^2$ and the map $\varphi \colon V_0 \to \aA^2$, $((x,y),[u:v])\mapsto (x,\frac vu)$ is an isomorphism. Similarly, the open subset 
$$
\widetilde{U_2}=\{([x:y:z],[u:v])\in B_{[0:0:1]}\PP^2\st z\ne 0, u\ne 0\}
$$ 
is Zariski-dense in $B_{[0:0:1]}\PP^2$ and the map $\widetilde{U_2} \to U_2 \simeq \aA^2$, 
$$
([x:y:z],[u:v])\mapsto [ux:v:uz]
$$ 
is an isomorphism.
Thus $B_{[0:0:1]}\PP^2$ is rational. Now remark that the map $\varphi \colon V_1=\{v\ne 0\} \to \aA^2$, $((x,y),[u:v])\mapsto (x,\frac uv)$ is also an isomorphism and the surface $B_{(0,0)}\aA^2$ is thus covered by two open subsets, both isomorphic to $\aA^2$.
We deduce that the surface $B_{[0:0:1]}\PP^2$ is covered by the three open subsets $U_0,U_1,\widetilde{U_2}=B_{[0:0:1]}U_2\simeq B_{(0,0)}\aA^2$ hence covered by four open subsets, both isomorphic to $\aA^2$. 
Up to affine transformation, we can define $B_P\PP^2$  for any $P\in\PP^2$ and it is now clear that the surface $B_P\PP^2$ is covered by a finite number of open subsets, each isomorphic to $\aA^2$. The same is clearly true for $\PP^1\times \PP^1$. It is also true for $Q_{3,1}$. Indeed, choose $3$ distinct real points $P_1,P_2,P_3$ of $Q_{3,1}$, and denote the open set $Q_{3,1}\setminus T_{P_i}Q_{3,1}$ by $U_i$, for $i=1,2,3$. Since the common intersection of the three projective tangent planes is a single point, that, moreover does not belong to $Q_{3,1}$, the subsets $U_1,U_2,U_3$ constitute an open affine covering of $Q_{3,1}$.

Let $X$ be an algebraic surface and $P$ be a real point of $X$. Assume that $P$ admits a neighborhood $U$ isomorphic to $\aA^2$ which is dense in $X$ (by Corollary~\ref{cor.rat} below we have in fact that if $X$ is rational, any real point of $X$ has this property), and define  
the blow-up of $X$ at $P$ to be the real algebraic surface obtained from $X\setminus \{P\}$ and $B_PU$ by gluing them along their common open subset $U\setminus \{P\}$. Then $B_PU\simeq B_PU_0$ is dense in $B_PX$ and contains a dense open subset isomorphic to $U_0\simeq \aA^2$. At this point, we admit  that this construction does neither depends on the choice of $U$, nor on the choice of the isomorphism between $U$ and $\aA^2$. See e.g. \cite[\S II.4.1]{Sc94} or \cite[Appendice A]{ma-gaz} for a detailed exposition. 

We get:
\begin{prop}
\label{prop.blowseq}%
Let $X_0$ be one of the surfaces $\PP^2$, $\PP^1\times \PP^1$ or $Q_{3,1}$. If
$$
X_n\stackrel{\pi_n}{\longrightarrow} X_{n-1} \stackrel{\pi_{n-1}}{\longrightarrow} \cdots \stackrel{\pi_{1}}{\longrightarrow} X_0
$$
is a sequence of blow-ups at real points, then $X_n$ is a real rational surface.
\end{prop}
\begin{proof}
Indeed, from Example~\ref{ex.first} and the comments above, any point $P\in X_i$ admits a neighborhood $U$ isomorphic to $\aA^2$ which is dense in $X_i$.
\end{proof}
Let $\pi \colon B_PX \to X$ be the blow-up of $X$ at $P$.
The curve $E_P= \pi^{-1}\{P\}$ is the \emph{exceptional curve} of the blow-up. 
We say that $B_PX$ is the blow-up of $X$ at $P$ and that $X$ is obtained from $B_PX$ by the \emph{contraction} of the curve $E_P$.

\begin{ex}
\label{ex.nonrealblow}%
Notice that if $P$ is a real point of $X$, the resulting blown-up surface gets an anti-holomorphic involution lifting the one of~$X$. If $P$ is not real, we can obtain a real surface anyway by blowing up both $P$ and $\overline{P}$:
let $U$ be an open neighborhood of $P$ which is complex isomorphic to $\aA^2(\CC)$ and define $B_{P,\overline{P}}X$ to be the result of the gluing of $X\setminus \{P,\overline{P}\}$ with both $B_PU$ and $B_{\overline{P}}\overline{U}$.
\end{ex}

\begin{rem}
In Example~\ref{ex.first}.3, the rational map $\pi_N$ decomposes into the blow-up of $Q_{3,1}$ at $N$, followed by the contraction of the strict transform of the curve $z=w$ (intersection of $Q_{3,1}$ with the tangent plane $T_NQ_{3,1}$), which is the union of two non-real conjugate lines. The rational map $\pi_N^{-1}$ decomposes into the blow-up of the two non-real points $[1:\pm i :0]$, followed by the contraction of the strict transform of the line $z=0$.
\end{rem}

The exceptional curve is a real rational curve isomorphic to $\PP^1$ whose real locus $E_P(\RR)$ is diffeomorphic to the circle $\sS^1$. Furthermore, the normal bundle of the smooth curve $E_P(\RR)$ in the smooth surface $B_PX(\RR)$ is nonorientable, thus $E_P(\RR)$ possesses a neighborhood diffeomorphic to the M\"obius band in $B_PX(\RR)$. Hence, topologically speaking, $B_PX(\RR)$ is obtained from $\xr$ through the following surgery (see Figure~\ref{fig.bup}): from $\xr$, remove a disk $D$ centered at $P$ and note that the boundary $\partial D$ is diffeomorphic to the circle~$\sS^1$, then paste a M\"obius band $M$, whose boundary $\partial M$ is also diffeomorphic to the circle $\sS^1$, to get $B_PX(\RR)$ which is then diffeomorphic to the connected sum (see e.g. \cite[Section~9.1]{Hi76}) 

$B_PX(\RR)\approx X(\RR)\#\PP^2(\RR)$.

\begin{figure}[ht] 
\centering
\includegraphics[scale=.7]{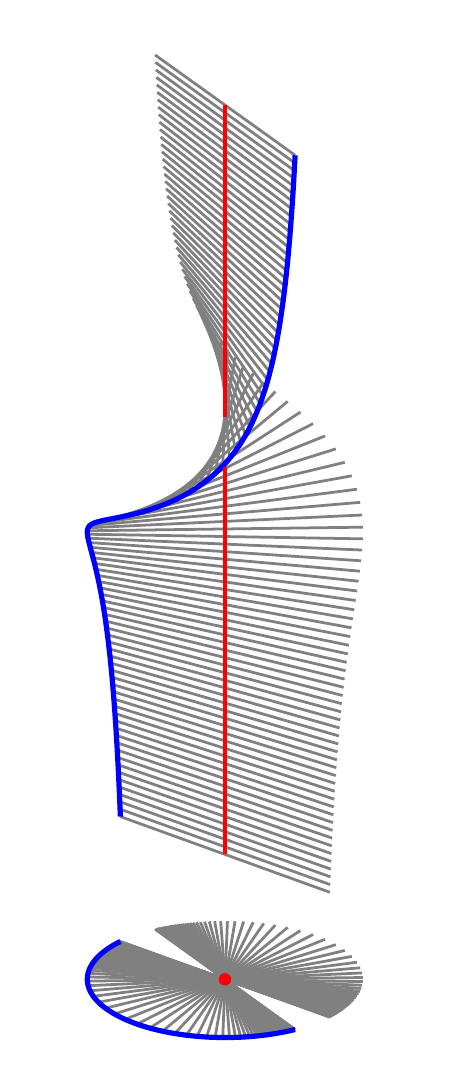}
\caption{The real locus of the exceptional curve is depicted by the vertical line.}
        \label{fig.bup}
\end{figure}
In particular 
\begin{equation}
\label{eq.klein}%
(B_P\PP^2)(\RR)\approx \PP^2(\RR)\#\PP^2(\RR)
\end{equation}
is the \emph{Klein bottle}.
From the classification of compact connected topological surfaces (see e.g. \cite[Theorem~9.3.10]{Hi76}), we know that any nonorientable compact connected topological surface $S$ is diffeomorphic to the connected sum of $g$ copies of the real projective plane $\PP^2(\RR)$:
$$
S\approx \PP^2(\RR)\#\PP^2(\RR)\#\dots \#\PP^2(\RR)\;.
$$

The positive integer $g$ is uniquely determined by $S$ and is called the \emph{genus} of $S$; for example the genus of $\PP^2(\RR)$ is $1$ and the genus of the Klein bottle is~$2$.

\begin{rem}
\begin{enumerate}

\item The uniqueness of $g$ follows from the fact that the abelianization of the fundamental group $\pi_1(S)$ is isomorphic to $\ZZ_2\oplus \ZZ^{g-1}$ if $S$ is a nonorientable surface of genus $g$ (the genus of a nonorientable surface is always positive). 
\item To be complete, recall that the genus of $\sS^2$ is $0$ and that 
an orientable surface $S$ of genus $g\geq 1$ is diffeomorphic to the connected sum of $g$ copies of the torus $\sS^1\times \sS^1$; the abelianization of $\pi_1(S)$ is isomorphic to $\ZZ^{2g}$. 
\item Let $S$ be an orientable surface of genus $g$, the abelianization of $\pi_1(\PP^2(\RR)\#S)$ is isomorphic to $\ZZ_2\oplus \ZZ^{2g}$. Hence the connected sum $\PP^2(\RR)\#S$ is a nonorientable surface of genus $2g+1$.
\end{enumerate}
\end{rem}

\subsection{Rational models}
Up to this point, from a given real rational surface, we worked out the topology of its real locus. We reverse now the point of view.

\begin{dfn}
Let $S$ be a compact connected topological surface. A real rational surface $X$ is a \emph{real rational model} of $S$ if the real locus is diffeomorphic to $S$: 
$$
\xr\approx S \;.
$$
\end{dfn}

The preceding observations and Examples~\ref{ex.first}.2 and \ref{ex.first}.3 above lead to the following consequence:
\begin{cor}
\label{cor.nor}%
Let $S$ be a compact connected topological surface. If $S$ is nonorientable, or orientable of genus $0$ or $1$, then $S$ admits a real rational model.
\end{cor}

A deep result of Comessatti \cite[p. 257]{Co14} states that the other topological surfaces do not have any real rational model:

\begin{thm}[Comessatti]
\label{thm.com}%
Let $X$ be a nonsingular 
real rational surface.
Then, if orientable, the real locus $\xr$ is diffeomorphic to the sphere $\sS^2$ or to the torus $\sS^1\times \sS^1$.
\end{thm}

Otherwise said: the real locus of a real rational surface is diffeomorphic to a sphere, a torus, or a nonorientable compact connected topological surface, and all these surfaces have real rational models.

A modern proof uses the Minimal Model Program for real algebraic surfaces as developed by Koll\'ar \cite[p. 206, Theorem.~30]{Ko01real} (see also \cite[Prop.~4.3]{Si89} 
for an alternative proof). In fact that approach gives us an even more precise statement.

Let $X$ and $Y$ be two real rational models of a given topological surface $S$. We will say that $X$ and $Y$ are \emph{isomorphic as real rational models} if their real loci $\xr$ and $\yr$ have isomorphic Zariski open neighborhoods in $X$ and $Y$, respectively. Equivalently, the surfaces $\xr$ and $\yr$ are \emph{birationally diffeomorphic}, that is: there is a diffeomorphism $f \colon \xr \to \yr$ which extends as a real birational map $\psi \colon X \to Y$ whose indeterminacy locus does not intersect $\xr$, and such that the indeterminacy locus of $\psi^{-1}$ does not intersect $\yr$. 

\begin{ex}
\label{ex.p2s2}%
Let $P$ be a real point of the sphere $\sS^2=Q_{3,1}(\RR)$. Then the blow-up $B_PQ_{3,1}$ at $P$ is a real rational model of the topological real projective plane $\PP^2(\RR)$. The projective plane $\PP^2$ is also a real rational model of $\PP^2(\RR)$ as well. Although the real algebraic surfaces $B_PQ_{3,1}$ and $\PP^2$ are not isomorphic, the stereographic projection induces a birational diffeomorphism from $B_PQ_{3,1}(\RR)$ onto $\PP^2(\RR)$ sending the exceptional curve to the line at infinity. The real rational surfaces $B_PQ_{3,1}$ and $\PP^2$ are therefore isomorphic real rational models of the topological surface $\PP^2(\RR)$.
\end{ex}

Collecting preceding observations: $\PP^1\times\PP^1$ is a real rational model of the torus $\sS^1\times \sS^1$, $Q_{3,1}$ is a real rational model of the sphere $\sS^2$ and if $S$ is a nonorientable topological surface of genus $g$, the blow-up $B_{P_1,\dots,P_g}Q_{3,1}$, where $P_1,\dots,P_g$ are $g$ distinct real points, is a real rational model of $S$:  
$$
B_{P_1,\dots,P_g}Q_{3,1}(\RR)\approx B_{P_1,\dots,P_g}\sS^2 \approx \PP^2(\RR)\#\dots \#\PP^2(\RR) \quad \text{($g$ terms)}.
$$

Using Koll\'ar's Minimal Model Program \cite[\textit{loc. cit.}]{Ko01real},
one can prove the following statement (compare \cite[Thm.~3.1]{BH07}):
\begin{thm}
\label{thm.classif}%
Let $S$ be a compact connected topological surface and $X$ be a real rational model of $S$.
\begin{enumerate}
\item If $S$ is nonorientable then $X$ is isomorphic as a real rational model to a real rational model of $S$ obtained from $Q_{3,1}$ by successively blowing up real points only.
\item  If $S$ is orientable then $X$ is isomorphic to $Q_{3,1}$ or $\PP^1\times \PP^1$ as a real rational model.
\end{enumerate}
\end{thm}

\begin{rem}
In statement 1. above, the real rational model obtained from $Q_{3,1}$ may \textit{a priori} include infinitely near points
\end{rem}

\begin{cor}
\label{cor.rat}%
Any (nonsingular) real rational surface is covered by a finite number of open subsets, each isomorphic to $\aA^2$.
\end{cor}

Therorem~\ref{thm.classif} clearly implies Comessatti's Theorem above, but it also highlights the importance of classifying real rational models of a given topological surface (compare \cite[Theorem~1.3 and comments following it]{Ma06}).
Surprisingly enough, all real rational models of a given topological surface turn out to be isomorphic as real rational models. This has been proved by Biswas and Huisman \cite[Thm.~1.2]{BH07}:
\begin{thm}
\label{thm.bh}%
Let $S$ be a compact connected topological surface. Then any two real rational models of $S$ are isomorphic.
\end{thm}

\begin{proof}[Proof of Theorem~\ref{thm.classif}]
Apply the Minimal Model Program to $X$ in order to obtain a sequence of blow-ups 
$$
X_n\stackrel{\pi_n}{\longrightarrow} X_{n-1} \stackrel{\pi_{n-1}}{\longrightarrow} \cdots \stackrel{\pi_{1}}{\longrightarrow} X_0
$$
analogous to the one of  Proposition~\ref{prop.blowseq} except that we allow also blow-ups at pairs of nonreal points as in Example~\ref{ex.nonrealblow} and that $X_0$ is now one of the following (see \cite[p. 206, Theorem.~30]{Ko01real}):
\begin{enumerate}
\item a surface with nef canonical bundle;
\item a conic bundle $p\colon X_0 \to B$ over a nonsingular real algebraic curve with an even number of real singular fibers, each of them being real isomorphic to $x^2+y^2=0$;
\item a "del Pezzo" surface: $\PP^2$, $Q_{3,1}$ or a del Pezzo surface with non connected real locus;
\end{enumerate}

Since $X$ is rational, $X_0$ is rational and we proceed trough a case by case analysis:

1. Recall that a line bundle is \emph{nef} if the degree of its restriction to any curve is nonnegative and that a rational surface cannot satisfy such a condition, see e.g. \cite[\S3.2]{BPV}. 

2. Since $X_0$ is rational, the base curve $B$ of the conic bundle $p$ is rational, that is $B$ is isomorphic to $\PP^1$. The image of the real locus of $X_0$ by $p$ is a finite set of intervals of $B(\RR)\approx \sS^1$; each interval corresponding to a connected component of $X_0(\RR)$. Over an interior point of such an interval, a fiber of $p$ is smooth and its real locus is diffeomorphic to the circle. Over a boundary point, the real locus is a single point and outside these intervals, the real locus of a fiber is empty.  Since $X_0(\RR)$ is connected and nonempty, the number of real singular fibers of the conic bundle is $0$ or  $2$. If it is $2$,   $X_0(\RR)$ is then diffeomorphic to $\sS^2$. In fact $X_0$ is isomorphic to $Q_{3,1}$ blown-up at a pair of nonreal points (see 
\cite[Example 2.13(3)]{blm2} for details). This reduces to the case when $X_0$ is isomorphic to $Q_{3,1}$. If there is no real singular fibers, $X_0$ is isomorphic to a $\PP^1$-bundle over $\PP^1$. By \cite[Theorem~1.3]{Ma06}, $X_0(\RR)$ is then birationally diffeomorphic to the Klein bottle $(B_P\PP^2)(\RR)$, see \eqref{eq.klein} p.~\pageref{eq.klein},
or to the torus $(\PP^1\times \PP^1)(\RR)$.
If $S$ is orientable we are done, since $X(\RR)$ is orientable too, and $X$ is obtained from $X_0$ by blowing up at nonreal points only. If $S$ is nonorientable, then $\xr$ is nonorientable either, and $X$ is obtained from $X_0$ by blowing up, at least, one real point.  If $X_0=\PP^1\times \PP^1$, a blow-up of $X_0$ at one real point is isomorphic to a blow-up of $\PP^2$ at two real points and then is isomorphic as a real rational model to some blow-up of $Q_{3,1}$. The remaining case, $X_0$ is the blow-up of $\PP^2$ at one point, reduces to the next case. 

3. The real locus of a real rational surface being connected, this rules out del Pezzo surfaces with non connected real locus.
 
It remains to show that the statement of the theorem holds if $X_0$ is isomorphic to $\PP^2$ or to $Q_{3,1}$. If $X_0$ is isomorphic to $\PP^2$, then by Example~\ref{ex.p2s2}, the stereographic projection reduces to the case $X_0$ is isomorphic to $Q_{3,1}$ as a real rational model. Now if $S$ is orientable, then $X(\RR)$ is orientable too, and like in the torus case, $X$ is obtained from $X_0$ by blowing up at nonreal points only. It follows that $X$ is isomorphic to $Q_{3,1}$ as a real rational model.
If $S$ is nonorientable, then $\xr$ is nonorientable either, and it is obtained from $Q_{3,1}$ by blowing up at real points. \end{proof}

\begin{ex}[A real del Pezzo surface with non connected real locus.]
\label{ex.dp2}%
The surface defined by the affine equation (the specific values of the coefficients correspond to Figure~\ref{fig.dp2})
$$
z^2+(8x^4+20x^2y^2-24x^2+8y^4-24y^2+16,25)=0
$$ 
is the double cover of the plane ramified over a quartic curve. This is a real minimal del Pezzo surface of degree $2$ whose real locus is diffeomorphic to the disjoint union of four spheres.
\begin{figure}[ht] 
\centering
\includegraphics[scale=.2]{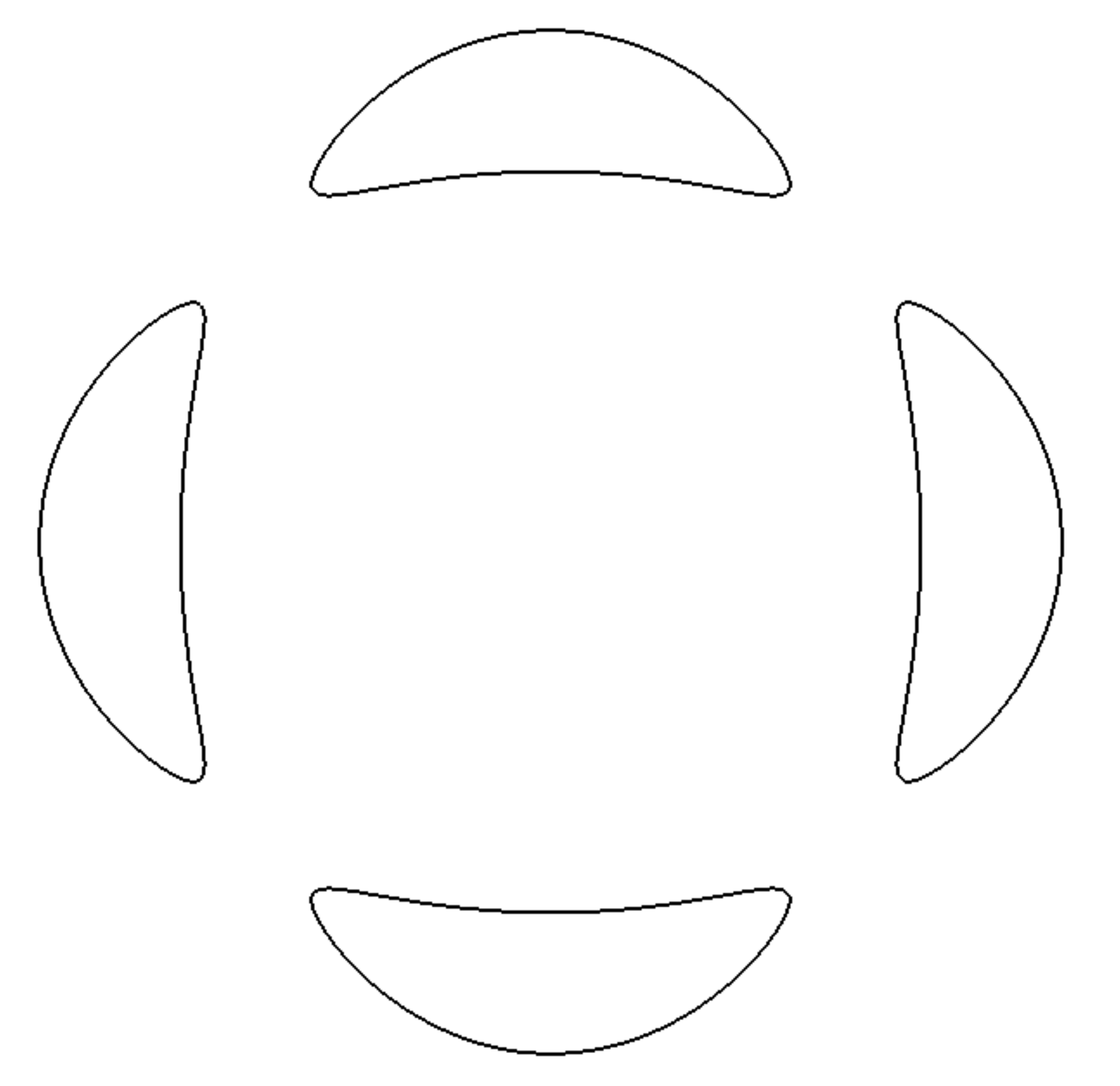}
\hskip 2cm
\includegraphics[scale=.35]{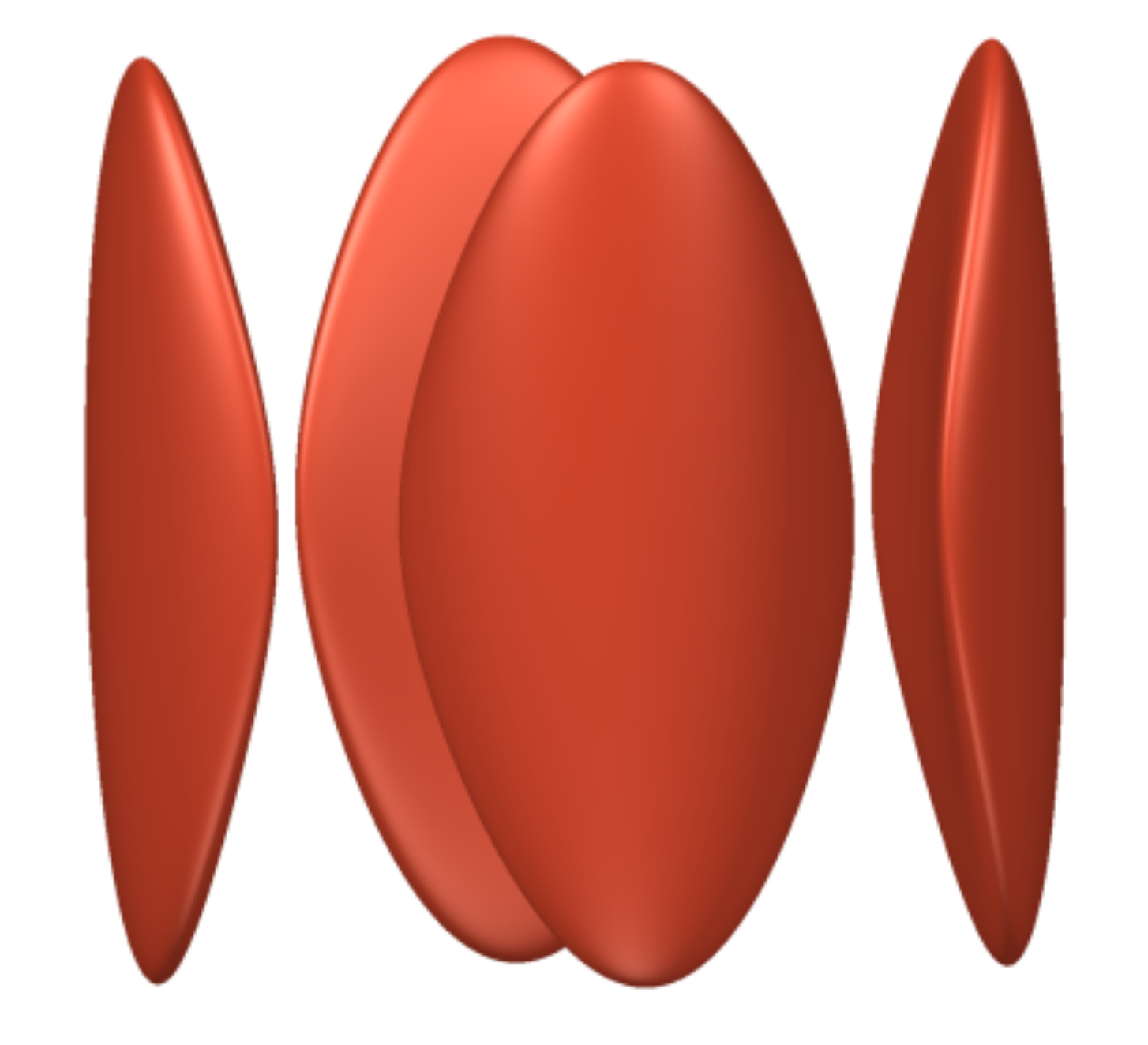}
\caption{On the left: the real locus of the real quartic curve given by $8x^4+20x^2y^2-24x^2+8y^4-24y^2+16,25=0$; on the right: the double plane ramified over it.}
        \label{fig.dp2}
\end{figure}
\end{ex}

\begin{proof}[Proof of Theorem~\ref{thm.bh}]
The proof which is given below is quite different from the one in \cite{BH07}; it is built on the fact that the group of self-birational diffeomorphisms of the sphere is \emph{infinitely transitive}, see Theorem~\ref{thm.hm} in the next section, this is the approach followed in \cite{hm3}.

The statement we want to prove is well-known if $S$ is the sphere or the torus. From the Comessatti's Theorem we may therefore assume that $S$ is a nonorientable surface.

First step. A crucial ingredient of the proof of Theorem~\ref{thm.bh} is the following. According to Theorem~\ref{thm.classif}, any real rational model $X$ of $S$ is isomorphic to a real rational model $Y$ of $S$ obtained from the sphere $\sS^2 = Q_{3,1}(\RR)$ by successively blowing up real points. This means that there is a sequence of blow-ups at real points
$$
Y=Y_n\stackrel{\pi_n}{\longrightarrow} \cdots \longrightarrow  Y_{2}\stackrel{\pi_{2}}{\longrightarrow}  Y_{1} \stackrel{\pi_{1}}{\longrightarrow} Y_0=Q_{3,1}\;.
$$

For simplicity, we describe this first step in the case $n=2$. Let $Q$ be a real point of $Q_{3,1}$ and let $P$ be a real point of the exceptional curve $E_Q$ of $\pi_1 \colon Y_1=B_QQ_{3,1}\to Q_{3,1}$.
If $Y_2=B_P(B_QQ_{3,1})$ is the blow-up of $Y_1$ at $P$, it is not \textit{a priori} clear that we can reduce to the case where $Y_2$ is the blow-up of $Q_{3,1}$ at two distinct points of $Q_{3,1}(\RR)=\sS^2$. One gets rid of this difficulty by using Example~\ref{ex.p2s2}. 
The algebraic surface $Y_1$ is a real rational model of $\PP^2(\RR)$ isomorphic to $\PP^2$, i.e. there is a birational diffeomorphism $f_Q\colon Y_1(\RR)\to\PP^2(\RR)$. 
Up to projectivities, we get moreover that for any real projective line $D$ of $\PP^2$, there is a birational diffeomorphism that maps the set of real points $E_Q(\RR)$ of the exceptional curve $E_Q$ to the real locus $D(\RR)$. Choose a real projective line $D(\RR)$ of $\PP^2(\RR)$ that does not contain the real point $f_Q(P)$ of $\PP^2$. 

There is a blow-up $Y_1'=B_{Q'}Q_{3,1}$ of the sphere at a real point, and a birational diffeomorphism $f_Q\colon Y_1'(\RR)\to \PP^2(\RR)$ mapping the real locus of the exceptional curve $E_{Q'}$ onto $D(\RR)$. Let $f={f_{Q'}}^{-1}\circ f_Q$ and $P'$
be the real point of $Y_1'$ corresponding to $P$ via the birational diffeomorphism $f \colon Y_1(\RR) \to Y_1'(\RR)$. 
Then the point $P'$ is not a point of the exceptional curve of the blow-up $\pi' \colon Y_1'=B_{Q'}Q_{3,1}\to Q_{3,1}$; which means that $\pi'$ maps isomorphically some affine neighborhood of $P'$ to an affine neighborhood of $\pi'(P')$. 

Since there is a birational diffeomorphism from $Y_1(\RR)$ to $Y_1'(\RR)$ that maps $P$ to $P'$, there is also a birational diffeomorphism from $Y_2(\RR)$ to $Y_2'(\RR)$, the real locus of the blow-up $Y_2'$ of $Y_1'$ at $P'$. Now, $Y_2'=B_{\pi'(P'),Q'}Q_{3,1}$ is the blow-up of $Q_{3,1}$ at 2 distinct real points, and is isomorphic as a real rational model to $Y_2=B_P(B_QQ_{3,1})$.

By an induction argument, one shows more generally that any real rational model $X$ of a nonorientable compact connected topological surface of genus $g$ is isomorphic to the blow-up $B_{P_1,\dots,P_g}Q_{3,1}$ where $P_1,\dots,P_g$ are $g$ distinct real points of the sphere.

Second step.
The second main ingredient of the proof is the fact that for any two $g$-tuples $(P_1,\dots,P_g)$ and $(Q_1,\dots,Q_g)$ of distinct elements of $\sS^2$, there is a a birational diffeomorphism $f\colon \sS^2\to \sS^2$ such that $f(P_j)=Q_j$ for all~$j$ (see Theorem~\ref{thm.hm} below).
Hence the blow-up $B_{P_1,\dots,P_g}Q_{3,1}$ is birationally diffeomorphic to the blow-up $B_{Q_1,\dots,Q_g}Q_{3,1}$. 
\end{proof}

\section{Automorphism groups of real loci}
The group of automorphisms of a compact complex algebraic variety is small: indeed, it is finite dimensional and even finite in most of the cases. And the same is true for the group of birational transformations of many varieties. 
On the other hand, the group $\Aut\bigl(\xr\bigr)$ of birational self-diffeomorphisms (also called automorphisms of $\xr$) of a real rational surface $X$ is quite big as the next results show.

\subsection{Transitivity}
Recall that a group $G$, acting on a set $M$, acts \emph{$n$-transitively} on $M$ if for any two $n$-tuples $(P_1,\dots,P_n)$ and $(Q_1,\dots,Q_n)$ of distinct elements of $M$, there is an element $g$ of $G$ such that $g\cdot P_j=Q_j$ for all $j$. 
The group $G$ acts \emph{infinitely transitively}
\footnote{In the literature, an \emph{infinitely} transitive group action is sometimes called a \emph{very} transitive action.} 
on $M$ if for every positive integer~$n$, its action is $n$-transitive on~$M$.
The next result is proved in \cite[Thm.1.4]{hm3}.

\begin{thm}
\label{thm.hm}%
Let $X$ be a nonsingular projective real rational surface. Then the group of birational diffeomorphisms $\Aut\bigl(\xr\bigr)$ acts infinitely transitively on $\xr$.
\end{thm}

\begin{proof}
In order to give an idea of the proof of the above theorem, let us show how one can construct many birational diffeomorphisms of the sphere $Q_{3,1}(\RR)\approx \sS^2$.
Let $I$ be the interval $[-1,1]$ in $\RR$. Let $\sS^1\subset \RR^2$ be the unit circle. Choose any smooth rational map $f
\colon I \to \sS^1$. This simply means that the two coordinate functions of $f$ are rational functions in one variable without poles in $I$.
Define a map $\phi_f \colon \sS^2 \to \sS^2$  ($\phi_f$ is called the \emph{twisting map} associated to $f$) by \label{p.twist}%
$$
\phi_f (x,y,z) = (f(z)\cdot(x,y),z)
$$
where $\cdot$ denotes complex multiplication in $\RR^2=\CC$; in other words, $f(z)\cdot(x,y)$ is a rotation of $(x,y)$ that depends algebraically on $z$.
The map $\phi_f$ is a birational self-diffeomorphism of~$\sS^2$. Indeed, its inverse is $\phi_g$ where $g\colon I \to \sS^1$ maps $z$ to the multiplicative inverse $(f(z))^{-1}$ of $f(z)$.
Now let $z_1,\dots,z_n$ be $n$ distinct points of $I$ and $\rho_1,\dots,\rho_n$ be elements of $\sS^1$. 
Then from Lagrange polynomial interpolation, there is a smooth rational map $f\colon I \to \sS^1$ such that $f(z_j)=\rho_j$ for $j=1,\dots,n$. 
The multiplication by $\rho_j$ is a rotation in the plane $z=z_j$, hence there exists a twisting map $\phi_f $ which \emph{moves} $n$ given distinct points $P_1,\dots,P_n$ on the sphere to $n$ another given points $R_1,\dots,R_n$ provided that each pair $P_j,R_j$ (same $j$) belong to an horizontal plane ($z=cst$). Let $(P_1,\dots,P_n)$ and $(Q_1,\dots,Q_n)$ be $n$-tuples of distinct elements of $\sS^2$.
To get a birational self-diffeomorphism mapping each $P_j$ to each $Q_j$, it suffices to consider two transversal families of parallel planes in order to get $n$ intersection points $R_j$, see Figure~\ref{fig.trans}. Then up to linear changes of coordinates, apply twice the preceding construction to get $2$ twisting maps, the first one mapping $P_j$ to $R_j$, $j=1,\dots,n$; the second one mapping $R_j$ to $Q_j$, $j=1,\dots,n$. Hence the composition of these twisting maps gives the desired birational self-diffeomorphism. 

\begin{figure}[ht]
\centering
\includegraphics[scale=1]{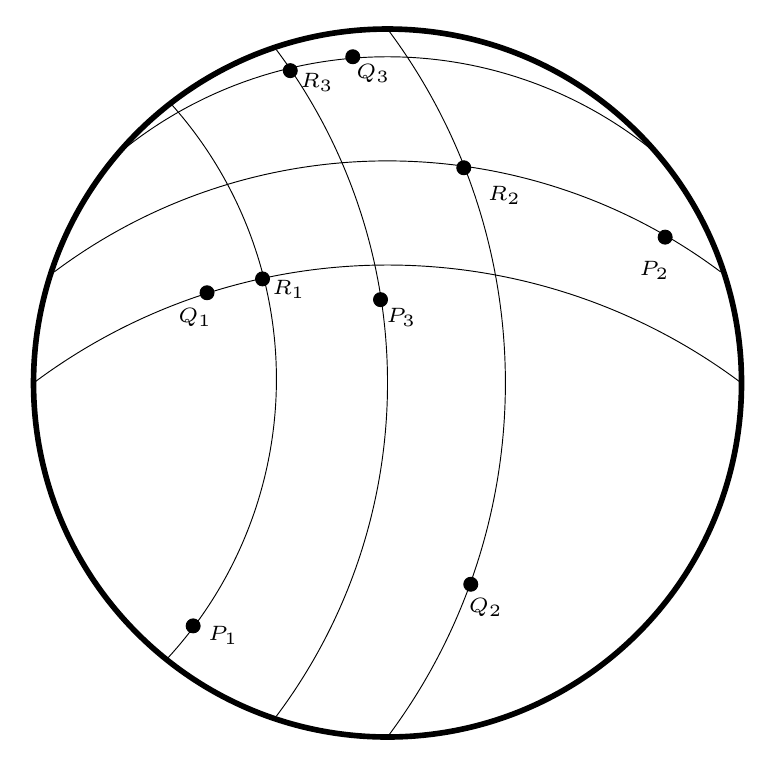}
\caption{The sphere $\sS^2$ with two sets of parallels.}
        \label{fig.trans}
\end{figure}
\end{proof}

\begin{rem}
By induction on the dimension, we can prove with this construction that in fact the group $\Aut\bigl(\sS^n\bigr)$ acts infinitely transitively on $\sS^n$ for $n>1$.
\end{rem}

Theorem~\ref{thm.hm} deals with real algebraic surfaces which are rational.
More generally, a real algebraic surface is \emph{geometrically rational}\label{geomrat} if the complex surface (that is the real surface forgetting the anti-holomorphic involution) contains a dense open subset complex isomorphic to $\aA^2(\CC)$. Clearly, a real rational surface is geometrically rational but the converse is not true. For example, the real del Pezzo surface of Example~\ref{ex.dp2} is a geometrically rational non rational surface. In the paper \cite[Thm.~1]{blm1}, the question of infinite transitivity of the automorphism group is settled for geometrically rational surfaces and in fact for all real algebraic surfaces. Below is one result of \textit{ibid}.

\begin{thm}
\label{thm.geomrat}%
Let $X$ be a real algebraic surface (smooth and projective as above). The group $\Aut\bigl(\xr\bigr)$ of birational diffeomorphisms is infinitely transitive on each connected component of $\xr$ if and only if $X$ is geometrically rational and the number of connected components satisfies $\#\pi_0(\xr)\leq 3$.
 \end{thm}
 
In the statement above,  the action of $\Aut\bigl(\xr\bigr)$ on $\xr$ is said to be infinitely transitive \emph{on each connected component} if for any pair of $n$-tuples of distinct points $(P_1,\dots,P_n)$ and $(Q_1,\dots,Q_n)$ of $\xr$ such that for each $j$, $P_j$ and $Q_j$ belong to the same connected component of $\xr$, there exists a birational diffeomorphism $f\colon \xr\to \xr$ such that $f(P_j)=Q_j$ for all~$j$.

\begin{rem}
The infinite transitivity of the automorphism groups of real algebraic varieties has been proved also for rational surfaces with mild singularities in \cite{hm4} (here, the infinite transitivity has to be understood on the regular part of the surface); and the question of infinite transitivity in the context of affine varieties is studied in \cite{kum1}.
\end{rem}

\subsection{Generators}
A closely related line of research studies generators of $\Aut\bigl(\xr\bigr)$ for various real rational surfaces $X$.
The classical Noether-Castelnuovo Theorem \cite{Ca01} (see also \cite[Chapter 8]{Al02} for a modern exposition of the proof) gives generators of the group $\Bir_\CC(\PP^2)$ of birational transformations of the complex projective plane. The group is generated by the  biregular automorphisms, which form the group $\Aut_\CC(\PP^2)= \PGL(3,\CC)$ of projectivities, and by the standard quadratic transformation 
$$
\sigma_0\colon [x:y:z]\dasharrow [yz:xz:xy].
$$

This result does not work over the real numbers. Indeed, recall that a \emph{base point} of a birational transformation is a (possibly infinitely near) point of indeterminacy; and note that two of the base points of the quadratic involution
$$
\sigma_1\colon [x:y:z]\dasharrow [y^2+z^2: xy: xz]
$$
are not real. Thus $\sigma_1$ cannot be generated by projectivities and $\sigma_0$. More generally, we cannot generate this way maps having nonreal  base-points.
Hence the group $\Bir_\RR(\PP^2)$ of birational transformations of the real projective plane is not generated by $\Aut_\RR(\PP^2)= \PGL(3,\RR)$  and $\sigma_0$.

The main result of \cite[Thm.~1.1]{blm2} is that $\Bir_\RR(\PP^2)$ is generated by $\Aut_\RR(\PP^2)$, $\sigma_0$, $\sigma_1$, and a family of birational maps of degree $5$ having only nonreal base-points:

\begin{ex}
\label{ex.quintic}%
Let $p_1,\bar{p_1},p_2,\bar{p_2},p_3,\bar{p_3}\in \PP^2$ be  three pairs of conjugated non-real points of $\PP^2$, not lying on the same conic. Denote by $\pi\colon X\to \PP^2$ the blow-up of the six points, it induces a birational diffeomorphism $\xr\to \PP^2(\RR)$. Note that $X$ is isomorphic to a smooth cubic surface in $\PP^3$, see e.g. \cite[Proposition~IV.9]{Be78}. The set of strict transforms of the conics passing through five of the six points provides   three pairs of non-real lines on the cubic, and the six lines are disjoint. The contraction of these six lines gives a birational morphism $\eta\colon X\to \PP^2$, inducing an isomorphism $\xr\to \PP^2(\RR)$, which contracts the curves onto three pairs of non-real points $q_1,\bar{q_1},q_2,\bar{q_2},q_3,\bar{q_3}\in \PP^2$; we choose the order so that $q_i$ is the image of the conic not passing through $p_i$.
The  map $\psi=\eta\pi^{-1}$ is a birational map $\PP^2\dasharrow \PP^2$ inducing a birational diffeomorphism $\PP^2(\RR)\to \PP^2(\RR)$. 

 Let $L\subset \PP^2$ be a general line of $\PP^2$. The strict transform of $L$ on $X$ by $\pi^{-1}$ has self-intersection $1$ and intersects the six curves contracted by $\eta$ into $2$ points (because these are conics). The image $\psi(L)$ has then six singular points of multiplicity $2$ and self-intersection $25$; it is thus a quintic passing through the $q_i$ with multiplicity $2$. The construction of $\psi^{-1}$ being symmetric as the one of $\psi$, the linear system of $\psi$ consists of quintics of $\PP^2$ having multiplicity $2$ at  $p_1,\bar{p_1},p_2,\bar{p_2},p_3,\bar{p_3}$. 
\end{ex}

The proof of Theorem~\ref{thm.gene} below is based on a extensive study of Sarkisov links. As a consequence, \cite{blm2} recovers the set of generators of $\Aut(\PP^2(\RR))$ given in \cite[Teorema II]{RV05} and the set of generators of $\Aut(Q_{3,1}(\RR))$ given in \cite[Thm.~1]{km1}. Before stating theses results, we define another family of birational maps.

\begin{ex}
\label{ex.cubic}%
Let $p_1,\bar{p_1},p_2,\bar{p_2}\in Q_{3,1}\subset \PP^3$ be two pairs of conjugated non-real points, not on the same plane of $\PP^3$. Let $\pi\colon X\to Q_{3,1}$ be the blow-up of these points. The non-real plane of $\PP^3$ passing through $p_1,\bar{p_2},\bar{p_2}$ intersects $Q_{3,1}$ onto a conic, having self-intersection $2$: two general different conics on $Q_{3,1}$ are the trace of hyperplanes, and intersect then into two points, being on the line of intersection of the two planes. The strict transform of this conic on $X$ is thus a $(-1)$-curve. Doing the same for the other conics passing through $3$ of the points $p_1,\bar{p_1},p_2,\bar{p_2}$, we obtain four disjoint $(-1)$-curves on $X$, that we can contract in order to obtain a birational morphism $\eta\colon X\to Q_{3,1}$; note that the target is $Q_{3,1}$ because it is a smooth projective rational surface of Picard number $1$. We obtain then a birational map $\psi=\eta\pi^{-1}\colon Q_{3,1}\dasharrow Q_{3,1}$ inducing an isomorphism $Q_{3,1}(\RR)\to Q_{3,1}(\RR)$.

Denote by $H\subset Q_{3,1}$ a general hyperplane section. The strict transform of $H$ on $X$ by $\pi^{-1}$ has self-intersection $2$ and has intersection $2$ with the $4$ contracted curves. The image $\psi(H)$ has thus multiplicity $2$  and   self-intersection $18$; it is then the trace of a cubic section. The construction of $\psi$ and $\psi^{-1}$ being similar, the linear system of $\psi$ consists of cubic sections with multiplicity $2$ at $p_1,\bar{p_1},p_2,\bar{p_2}$.
\end{ex}

\begin{thm}
\label{thm.gene}%
\begin{enumerate}
\item The group $\Bir_\RR(\PP^2)$ is generated by $\Aut_\RR(\PP^2)$, $\sigma_0$, $\sigma_1$, and by the quintic transformations of $\PP^2$ defined in Example~\ref{ex.quintic}.

\item The group $\Aut(\PP^2(\RR))$ is generated by 
$$
\Aut_\RR(\PP^2)=\PGL(3,\RR)
$$ 
and by the quintic transformations of $\PP^2$ defined in Example~\ref{ex.quintic}.

\item The group $\Aut(Q_{3,1}(\RR))$ is generated by 
$$
\Aut_\RR(Q_{3,1})=\PO(3,1)
$$  
and by the cubic transformations defined in Example~\ref{ex.cubic}.
\end{enumerate}
\end{thm}

As remarked in \cite[Proposition~5.6]{blm2}, the twisting maps defined in the proof of Theorem~\ref{thm.hm} are compositions of twisting maps  of degree $1$ and $3$. And in the latter case the twisting maps belong to the set of cubic transformations defined in Example~\ref{ex.cubic}.

A new set of generators, completing the list for "minimal" real rational surfaces is also given \cite[Thm.~1.4]{blm2}:
\begin{thm}
The group $\Aut((\PP^1\times\PP^1)(\RR))$ is generated by 
$$
\Aut_\RR(\PP^1\times\PP^1)\cong\PGL(2,\RR)^2\rtimes \ZZ/2\ZZ
$$  
and by the birational involution
$$
\tau_0\colon ([x_0:x_1],[y_0:y_1])\dasharrow ([x_0:x_1],[x_0y_0+x_1y_1:x_1y_0-x_0y_1]).
$$
\end{thm}

\begin{rem}
Here is an analogous statement in the complex setting (see \cite{Is79,Is85}).
The group $\Bir(\PP^1\times\PP^1)$ is generated by 
$$
\Aut_\CC(\PP^1\times\PP^1)\cong\PGL(2,\CC)^2\rtimes \ZZ/2\ZZ
$$  
and by the birational involution
$$
e\colon ([x_0:x_1],[y_0:y_1])\dasharrow ([x_0:x_1],[x_0y_1:x_1y_0]).
$$
\end{rem}

\begin{rem}
For the interested reader, we put the stress on recent "real" results on Cremona groups: a rather complete classification of real structures on del Pezzo surfaces \cite{Ru02}; the study of the structure of some subgroups of the  real Cremona group \cite{Ro14} and \cite{Zi14}.
\end{rem}

\section{Approximation of differentiable maps by algebraic maps}
\subsection{Real algebraic models}
We have defined \emph{real rational models} of topological surfaces in Section~\ref{sec.real.rat}. More generally, let $M$ be a compact $\cC^\infty$-manifold without boundary; a real algebraic manifold $X$ is a \emph{real algebraic model} of $M$ if the real locus is diffeomorphic to $M$: 
$$
\xr\approx M \;.
$$

Clearly, a topological surface admitting a real rational model admits also a real algebraic model but the converse is not true. Indeed, by Comessatti's Theorem~\ref{thm.com}, an orientable surface of genus $g\geq2$ does not admit a real rational model but one of the two real algebraic surfaces given by the affine equations $z^2=\pm f(x,y)$, where $f$ is the product of equations of $g+1$ well chosen circles, is a real algebraic model\footnote{In fact, such a surface is not a manifold since it has nonreal singular points; but it is easy to get a manifold by "resolution" of these singular points or by a small deformation of the plane curve $f(x,y)=0$.} of a genus $g$ orientable surface. 

Note that the latter construction together with Corollary~\ref{cor.nor} proves that any compact topological surface admits a real projective model. In higher dimension, a striking theorem of Nash \cite{Nash52} improved by Tognoli \cite{To73} is the following (see \cite[Chapter 14]{BCR} for a proof):

\begin{thm}[Nash 1952, Tognoli 1973]
\label{thm.nash-to}%
Let $M$ be compact $\cC^\infty$-manifold without boundary, then there exists a nonsingular projective real algebraic variety $X$ whose real locus is diffeomorphic to $M$:
$$
M\approx X(\RR).
$$
\end{thm}

One of the most famous application of the Nash Theorem is the Theorem of Artin-Mazur \cite{AM65} below.
For any self-map $f \colon M \to M$, denote by $N_\nu(f)$ the number of \emph{isolated} periodic points of $f$, of period $\nu$ (\textit{i.e.}, the number of isolated fixed points of $f^\nu$).
\begin{thm}
Let $M$ be a compact $\cC^\infty$-manifold\footnote{In fact, the following results are valid for any $\cC^k$-regularity, $k=1,\dots,\infty$.} without boundary, and let $F(M)$ be the space of $\cC^\infty$-self maps of $M$ endowed with the $\cC^\infty$-topology. There is a dense subset $\cE\subset F(M)$ such that if $f\in \cE$, then $N_\nu(f)$ grows at most exponentially (as $\nu$ varies through the positive integers).
\end{thm}

\subsection{Automorphism groups}

The proof of Artin-Mazur's Theorem uses the fact that any $\cC^\infty$-self map of $M$ has an approximation by Nash morphisms, see e.g. \cite[Chapter 8]{BCR}. We want to stress here a big gap between Nash diffeomorphisms and  birational diffeomorphisms. A diffeomorphism which is also a rational map without poles on the real locus is a Nash diffeomorphism but not necessarily a birational diffeomorphism. Indeed, the converse diffeomorphism is not always rational. For instance the map $x \mapsto x+x^3$ is a Nash self-diffeomorphism of $\RR$ but it is not birational since the converse map has radicals. This is a consequence of the fact that the Implicit function Theorem holds in analytic setting but does not hold in the algebraic setting.

The question has been raised whether the group $\Aut\bigl(\xr\bigr)$ is dense in the group $\Diff\bigl(\xr\bigr)$ of all self-diffeomorphisms of $\xr$, for a real rational surface $X$. This turns out to be true and has been proved in \cite{km1}.

\begin{thm}\cite[Theorem~4]{km1}\ \hphantom{)}\
\label{thm.km1}%

Let $S$ be a compact connected topological surface and $\Diff(S)$ its group of self-diffeomorphisms endowed with the $\cC^\infty$-topology.
If $S$ is nonorientable or of genus $g(S)\leq 1$,  then there exists a real rational model $X$ of $S$ such that
$$
\overline{\Aut\bigl(\xr\bigr)}=\Diff\bigl(\xr\bigr)
$$

i.e. $\Aut\bigl(\xr\bigr)$ is a dense subgroup of $\Diff\bigl(\xr\bigr)$ for the $\cC^\infty$-topology.
\end{thm}

\begin{rem}
\label{rem.nondensity}%
If $S$ is orientable of genus $g(S)\geq 2$, then for any real algebraic model $X$ of $S$, we have
$\overline{\Aut\bigl(\xr\bigr)}\ne\Diff\bigl(\xr\bigr)$.
Let $X$ be a real algebraic surface with orientable real locus. Then following up the classification of surfaces (see e.g. \cite{BPV,Si89}):
if $X$ is geometrically rational or ruled, then $\xr \approx \sS^2$ or $\xr \approx \sS^1\times \sS^1$;  
if $X$ is  K3 or abelian, then $\Aut(\xr)$ preserves a volume form, hence density does not hold;
if $X$ is Enriques or bi-elliptic, it admits a finite cover by one surface in the former case, hence density does not hold;
if $X$ is properly elliptic, then $\Aut(\xr)$  preserves a fibration, hence density does not hold;
if $X$ is of general type, then $\Aut(\xr)$  is finite, hence density does not hold.
Summing up, if $g(S)>1$, then for any real algebraic model, density does not hold.
\end{rem}

\begin{rem}
Thanks to \cite[Thm.~2]{Lu77}, see below, the group $\Aut\bigl(\sS^n\bigr)$ is a dense subgroup of $\Diff\bigl(\sS^n\bigr)$  for $n>1$.
\end{rem}

\begin{proof}[Sketch of proof]
Any such topological surface admits a real rational model which is $\PP^1\times \PP^1$, $Q_{3,1}$ or the blow-up $B_{P_1,\dots,P_g}Q_{3,1}$ where $P_1,\dots,P_g$ are $g$ distinct real points of the sphere. Leaving aside the torus case for simplicity, we start with a theorem of Lukacki\u{\i} to the effect that the density holds for the sphere. Recall that for a given topological group $G$, the connected component containing the identity element is called the \emph{identity component} of $G$ and is denoted by $G_0$. 
The paper \cite[Thm.~2]{Lu77}  proves indeed that for any integer $n>1$, the topological group $\SO(n+1,1)$ is a maximal
closed subgroup of the identity component $\Diff_0(\sS^n)$ of $\Diff(\sS^n)$, meaning that any topological subgroup of the topological group $\Diff_0(\sS^n)$ that contains strictly $\SO(n+1,1)$ is dense in $\Diff_0(\sS^n)$. 
Consequently, any subgroup of $\Diff(\sS^n)$ that contains   strictly
$\oO(n+1,1)$ is dense in $\Diff(\sS^n)$. Thanks to this argument, we prove that $\overline{\Aut\bigl( \sS^n\bigr)}=\Diff\bigl( \sS^n\bigr)$ for $\xr\approx \sS^n$, $n>1$. For the case we are concerned with, the group $\oO(3,1)$ together with any nontrivial twisting map\footnote{See p.~\pageref{p.twist}} of $\sS^2$ generate a dense subgroup of $\Diff(\sS^2)$.

The remaining cases are the nonorientable surfaces $B_{P_1,\dots,P_g}Q_{3,1}(\RR)$. Let $X=B_{P_1,\dots,P_g}Q_{3,1}$. The proof is in three steps:

\begin{enumerate}
\item (Marked points).
Let $f$ be a self-diffeomorphism of $\sS^2$.
Let $f'$ be a birational self-diffeomorphism of $\sS^2$ close to $f$ given by density. Then the point $Q_j=f'(P_j)$ is close to $P_j$ for $j=1,\dots,g$.
By Theorem~\ref{thm.hm}, we get a birational self-diffeomorphism $h$ such that $P_j=h(Q_j)$ for $j=1,\dots,g$. Moreover, the construction of such a $h$ shows that $h$ is close to identity. Thus, starting with a map $f'$ closer to $f$ if needed, we get that the group 
$\Aut(\sS^2, P_1,\dots, P_g)$ of birational self-diffeomorphisms of $\sS^2$ fixing each $P_j$ is dense in the group 
 $\Diff(\sS^2, P_1,\dots, P_g)$ of self-diffeomorphisms of $\sS^2$ fixing each $P_j$. 
 
\item (Identity component).

The infinite transitivity of $\Aut(\xr)$ gives many birational diffeomorphisms from $\xr$ to the sphere blown-up at $g$ points and in particular 
there is a finite open cover $\xr = \cup_j W_j$ 
such that for every $j$ 
there are $g$ distinct points
$P_{1j},\dots,P_{gj}\in \sS^2$ and a birational diffeomorphism
$f_j\colon \xr\to B_{P_{1j},\dots,P_{gj}}\sS^2$ such that $f_j(W_j)$ 
avoids the exceptional locus of $\pi_j\colon B_{P_{1j},\dots,P_{gj}}\sS^2 \to \sS^2$.  

Let $\phi\in \Diff_0(\xr)$.
By the Fragmentation Lemma, see \cite[Lemma~3.1]{PS70},  we can write $\phi=\phi_1\circ \dots \circ \phi_r$ such that
each $\phi_j$ is the identity outside $W_j$.
By use of $f_j$, each $\phi_j$ descends to $\phi_j'\in \Diff_0(\sS^2,P_{1j},\dots,P_{gj})$. By the previous point, we can consider an approximation of $\phi_j'$ in $\Aut_0(\sS^2,P_{1j},\dots,P_{gj})$ and lift it to $\xr$. So far we deduce the density of the identity component $\Aut_0(\xr)$ in  $\Diff_0(\xr)$.

\item (Mapping class group).
This is the main step. To get the conclusion we use the fact that  the \emph{modular group} 
$$
\Mod(\xr)=
 \Diff(\xr)/\Diff_0(\xr)
 $$ 
 (also called the \emph{mapping class group}) is generated by birational self-diffeomorphisms of $\xr$, see Theorem~\ref{thm.mod} below.
\end{enumerate} 
 \end{proof} 

Let $X$ be a real algebraic model of a topological surface $S$, then
the modular groups $\Mod(S)$ and $\Mod(\xr)$ are isomorphic.

\begin{thm}\cite[Theorem~27]{km1}
\label{thm.mod}%

Let $S$ be a nonorientable compact connected topological surface. 
Then there exists a real rational model $X$ of $S$ such that the group homomorphism
$$
\pi \colon \left\{
\begin{array}{ccc}
\Aut(\xr) &\longrightarrow &\Mod(\xr) \\
f  &\longmapsto &[f]
\end{array}
\right.
$$
is surjective. 
\end{thm}
\begin{proof}
To give an idea of the proof, here is the construction of a non-trivial generator of the modular group realized by a birational automorphism.
By a famous theorem of Dehn \cite{De38}, when $S$ orientable, $\Mod(S)$ is generated by Dehn twists, see below. When $S$ nonorientable,  Dehn twists generate an index 2 subgroup of $\Mod(S)$, and we need another kind of generator called a \emph{cross-cap slide}, see \cite{Li65} or \cite[Section~24]{km1}.

Let $S$ be any surface and
$C\subset S$ a simple closed smooth curve such that
$S$ is orientable along $C$. Cut $S$ along $C$, rotate one side
around once completely and glue the pieces back together.
This defines a  diffeomorphism $t_C$ of $S$, see
Figure~\ref{fig.dehn}.

\begin{figure}[ht]
\centering
\includegraphics[scale=.4]{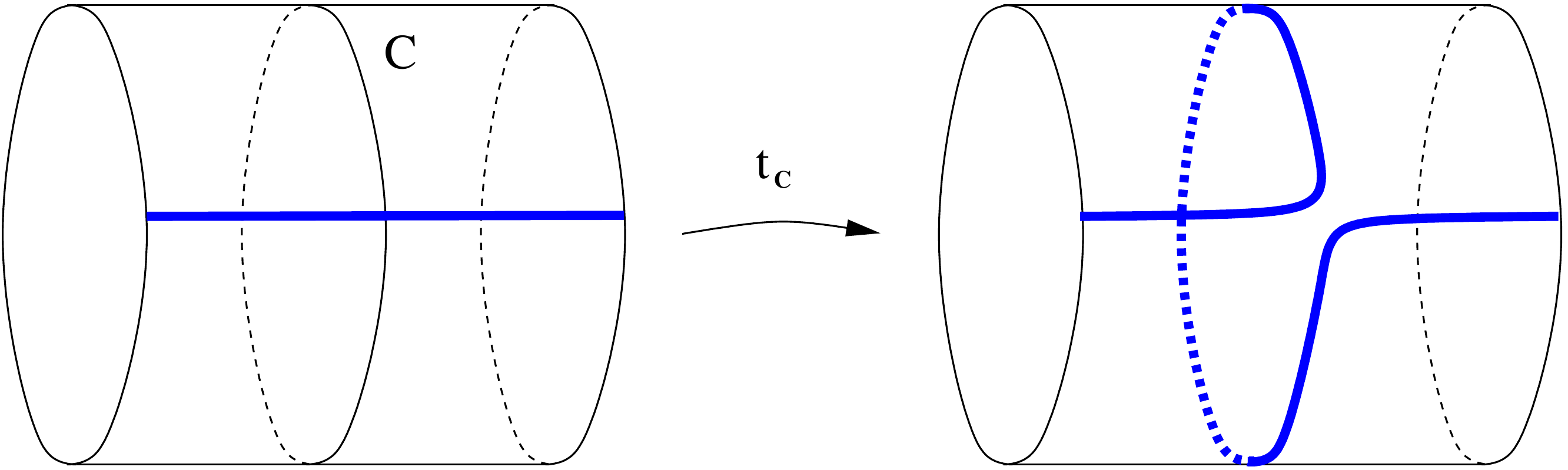}
\caption{The effect of the Dehn twist around $C$ on a curve.}
        \label{fig.dehn}
\end{figure}

The inverse
$t^{-1}_C$ corresponds to rotating one side the other way.
Up to isotopy, the pair $\{t_C, t^{-1}_C\}$ does not depend on the 
choice of $C$ or the rotation.
 Either of $t_C$ and $ t^{-1}_C$ is called a {\it Dehn twist} using $C$.
On an oriented surface, with $C$ oriented, one can
make a sensible distinction
between $t_C$ and $ t^{-1}_C$. This is less useful in the
non-orientable case.

Let $S$ be a nonorientable surface of genus $g$ and $X$ be the blow-up of $\sS^2$ at $g$ points $P_1,\dots,P_g$ a real rational model of $S$. Here is the construction of a Dehn twist in the easiest case. Let  $C^*\subset \sS^2$ be a smooth curve passing through none of the points $P_j$. After applying a  suitable automorphism of $\sS^2$, we may assume that $C^*$ is the big circle $(z=0)$.
Consider the map $g \colon [-1,1] \to \sS^1$ where $g(t)$ is the identity for
$t \in [-1,-\varepsilon]\cup[\varepsilon,1]$ and multiplication by $g(t)$ is the rotation by 
angle $\pi(1 + t/\varepsilon)$ for $t \in [-\varepsilon,\varepsilon]$. 
Let  $f \colon [-1,1] \to \sS^1$ be an algebraic approximation of $g$ such that the corresponding twisting map (cf. p.~\pageref{p.twist}) $\psi_f\colon \sS^2 \to \sS^2$ is the identity at the points $P_i$.
Then the lift of $\psi_f$ to $\xr=B_{P_1,\dots,P_g}\sS^2$ is a birational diffeomorphism of $\xr$; it's an algebraic realization of the Dehn twist using the lift $C$ of the curve $C^*$.
\end{proof}

It is straightforward to see that any element of the modular group is realized by a birational automorphism also in the case    $S\approx\sS^2$, whose modular group is isomorphic to $\ZZ_2$, 
and $S\approx\sS^1\times \sS^1$, whose modular group is isomorphic to $\GL(2,\ZZ)$\footnote{See e.g. \cite[Theorem~2.5]{FM12} for a computation of the modular group of \emph{orientation-preserving} diffeomorphisms.} and is realized by the group of monomial transformations. Thus any surface $S$ admitting a real rational model satisfies the statement of Theorem~\ref{thm.mod}. 

A byproduct of the proof of Theorem~\ref{thm.km1} is that $\Aut\bigl(\xr\bigr)$ is dense in $\Diff\bigl(\xr\bigr)$
when $X$ is a geometrically rational surface with $\#\pi_0(\xr)=1$ (or equivalently when $X$ is rational, see \cite[Corollary~VI.6.5]{Si89}). 
In  \cite{km1}, it is said that $\#\pi_0(\xr)=2$ is probably the only other case where the density holds, but his case remains open nowadays.
Summing up the known results in this direction, see~\cite{km1,blm1}, we get for a smooth real projective~$X$:

\begin{itemize}
\item
 If $X$ is not a geometrically rational surface, then 
$$
\overline{\Aut\bigl(\xr\bigr)}\ne\Diff\bigl(\xr\bigr)\;;
$$
\item If $X$ is a geometrically rational surface, then
\begin{itemize}
\item If $\#\pi_0(\xr)\geq 5$, then $\overline{\Aut\bigl(\xr\bigr)}\ne\Diff\bigl(\xr\bigr)$;
\item If $i=3,4$, there exists smooth real projective surfaces $X$ with $\#\pi_0(\xr)=i$ such that $\overline{\Aut\bigl(\xr\bigr)}\ne\Diff\bigl(\xr\bigr)$;
\item 
if  $\#\pi_0(\xr) =1$,
then $\overline{\Aut\bigl(\xr\bigr)}=\Diff\bigl(\xr\bigr)$. 
\end{itemize}
\end{itemize}

Note that the study of automorphism groups of other real algebraic surfaces than the rational ones has been developed from the point of view of topological entropy of automorphisms by several authors. In particular, 
Moncet \cite{Mo12} defines the \emph{concordance} $\alpha(X)$ for a real algebraic surface $X$ which is a number between $0$ and $1$ with the property that $
\overline{\Aut_\RR(X)}\ne\Diff\bigl(\xr\bigr)
$ as soon as $\alpha(X)>0$. (Notice that $\Aut_\RR(X)$ is the subgroup of $\Aut\bigl(\xr\bigr)$ of real automorphisms of the real algebraic surface $X$.) More precisely, when $X$ is a K3 surface, the groups $\Aut_\RR(X)$ and $\Bir_\RR(X)$ coincide and the non-density result when $\alpha(X)>0$ is stronger than the one of Remark~\ref{rem.nondensity}: actually, the group $\Aut_\RR(X)$ is discrete in the group of diffeomorphisms preserving the volume form given by the triviality of $K_X$. Note also that there exists K3 surfaces with infinite groups of automorphisms, cf. e.g. \cite{Mo12}.

We conclude this section by an important application of the Density Theorem~\ref{thm.km1}.
\begin{dfn}
For a differentiable manifold $M$, let $C^{\infty}(\sS^1,M)$ denote the space 
of all $C^{\infty}$ maps
 from $\sS^1$ to $M$, endowed with the $C^{\infty}$-topology.

Let $X$ be a smooth real algebraic variety and
$C\subset X$ a  rational curve.
By choosing any isomorphism of its normalization $\bar C$ with the
plane conic  $(x^2+y^2=z^2)\subset \PP^2$, we get a
$C^{\infty}$ map  $\sS^1\to \xr$ whose image coincides with
$C(\RR)$, aside from its isolated real singular points.

Let $f\colon L\hookrightarrow \xr$ be an embedded circle. 
We say that $L$ \emph{admits a $\cC^\infty$-approximation by smooth 
 rational   curves} if every neighborhood of $f$ in $C^{\infty}\bigl(\sS^1,\xr\bigr)$ 
contains a map derived as above from a  rational curve with no isolated real singular points.
\end{dfn}
\begin{thm}\cite{km2}
An embedded circle in a nonsingular real rational variety admits a $\cC^\infty$-approximation by smooth rational curves if and only if is is not diffeomorphic to a null-homotopic circle on a $2$-dimensional torus.
 \end{thm}

\section{Regulous maps}

In full generality the problem of approximation of differentiable maps by algebraic maps is still open. For instance, the existence of algebraic representatives of homotopy classes of continuous maps between spheres of different dimension does not have a complete solution nowadays, see \cite[Chapter~3]{BCR}. Here is an example of a result in this direction: if $n$ is a power of $2$, and if $p<n$, then any polynomial map from $\sS^n$ to $\sS^p$ is constant, cf. \cite[Thm.~13.1.9]{BCR}.

In \cite{Ku09}, Kucharz introduces the notion of \emph{continuous rational maps}  generalizing algebraic maps between real algebraic varieties. The particular case of continuous rational functions has also been studied by Koll\'ar very recently, see Koll\'ar-Nowak 
\cite{KN-regul}.
Continuous rational maps between nonsingular\footnote{In the singular case, the two notions may differ, see \cite{KN-regul}.} real algebraic varieties are now often called \emph{regulous maps} following \cite{regul}. 

Let $X$ and $Y$ be irreducible nonsingular real algebraic varieties whose sets of real points are Zariski dense. A \emph{regulous map} from $\xr$ to $\yr$ is a rational map $f\colon X \dasharrow Y$ with the following property. Let $U\subset X$ be the domain of the rational map $f$. The restriction of $f$ to $U(\RR)$ extends to a continuous map from $\xr$ to $\yr$ for the euclidean topology.
Kucharz shows that all homotopy classes can be represented by regulous maps \cite[Thm. 1.1]{Ku09}.

\begin{thm}
Let $n$ and $p$ be nonzero natural integers. Any continuous map from $\sS^n$ to $\sS^p$ is homotopic to a regulous map.
\end{thm}

In fact the statement is more precise: Let $n,p$ and $k$ be natural integers, $n$ and $p$ being nonzero. Any continuous map from $\sS^n$ to $\sS^p$ is homotopic to a $k$-regulous map. see below.

The paper \cite{regul} sets up foundations of a regulous geometry: algebra of regulous functions and regulous topologies. Here is a short account. Recall that a rational function~$f$ on~$\RR^n$ is called a \emph{regular function} on~$\RR^n$ if $f$ has no pole on~$\RR^n$. For instance, the rational function~$f(x)=1/(x^2+1)$ is regular on~$\RR$.  The set of regular functions on~$\RR^n$ is a subring of the field~$\RR(x_1,\dots,x_n)$ of rational function on~$\RR^n$.
A \emph{regulous function} on~$\RR^n$ is a real valued function defined at any point of~$\RR^n$, which is continuous for the euclidean topology and whose restriction to a nonempty Zariski open set is regular. A typical example is the function
$$
f(x,y)=\frac{x^3}{x^2+y^2}
$$ which is regular on~$\RR^2\setminus\{0\}$ and regulous on the whole~$\RR^2$. Its graph is the canopy of the famous Cartan umbrella, see Figure~\ref{fig.cartan}. 
\begin{figure}[ht]
\centering
\includegraphics[height =6cm]{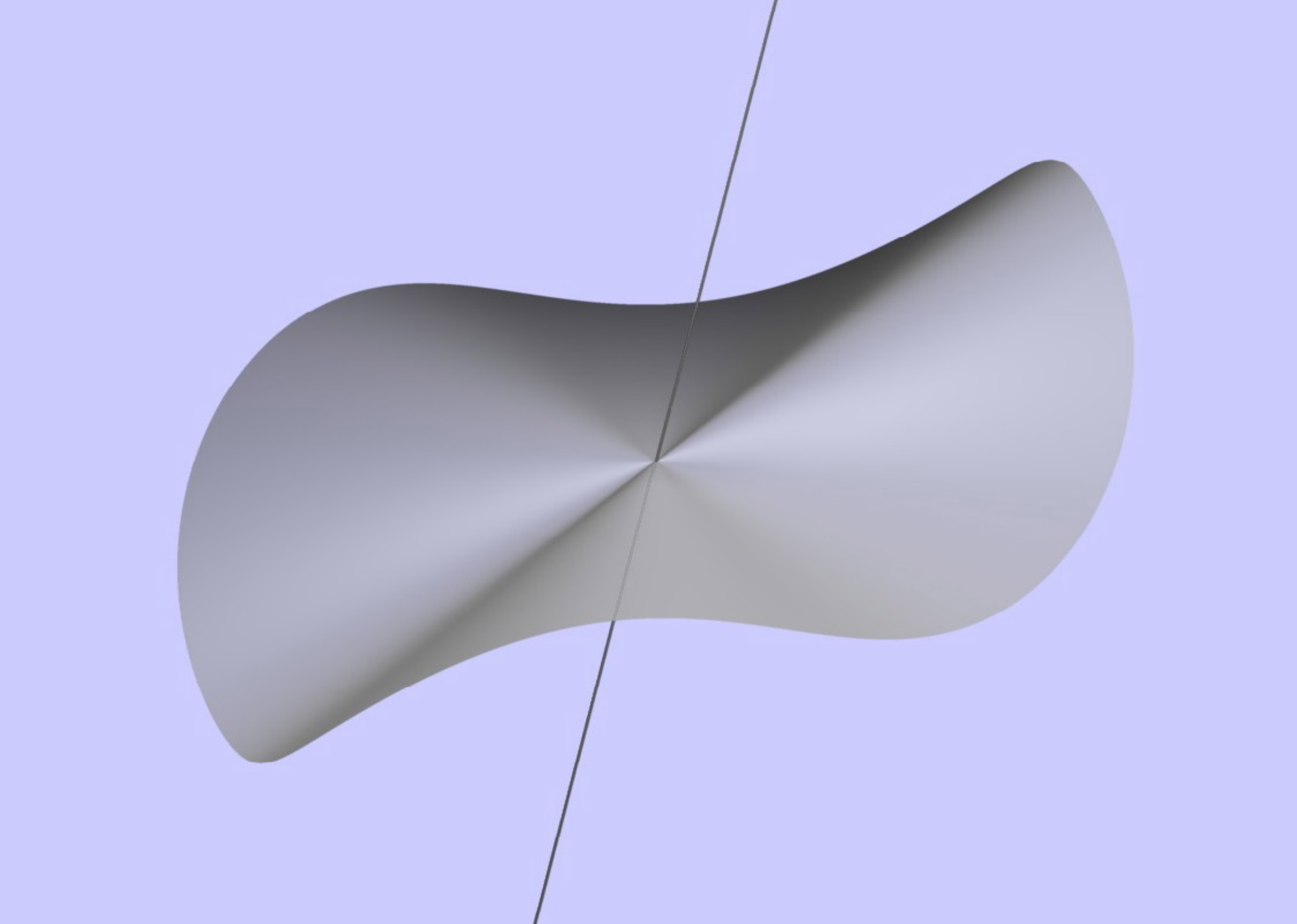}
\caption{The Cartan umbrella: $z(x^2+y^2)=x^3$.}
        \label{fig.cartan}
\end{figure}
The set of regulous functions on~$\RR^n$ is a subring $\cR^0(\RR^n)$ of the field~$\RR(x_1,\dots,x_n)$. More generally, a function defined on $\RR^n$ is \emph{$k$-regulous}, if it is at the same time, regular on a nonempty Zariski open set, and of class~$\cC^k$ on~$\RR^n$. Here, $k\in\NN\cup \{\infty\}$.  For instance, the function
$$
f(x,y)=\frac{x^{3+k}}{x^2+y^2}
$$ is $k$-regulous on $\RR^2$ for any natural integer $k$. We can prove that an $\infty$-regulous function on~$\RR^n$ is in fact regular (the converse statement is straightforward) and we get an infinite chain of subrings:
$$
\cR^\infty(\RR^n)\subseteq\cdots\subseteq\cR^2(\RR^n)\subseteq\cR^1(\RR^n)
\subseteq\cR^0(\RR^n)\subseteq\RR(x_1,\dots,x_n).
$$
where $\cR^k(\RR^n)$ denotes the subring of~$\RR(x_1,\dots,x_n)$ consisting of $k$-regulous functions.

The $k$-regulous topology is the topology whose closed sets are zero sets of $k$-regulous functions.
Figure~\ref{fig.horned} represents a "horned umbrella" which is the algebraic subset of $\RR^3$ defined by the equation $x^2+y^2\bigl((y-z^2)^2+yz^3\bigr)=0$.

It is irreducible for the $\infty$-regulous topology, but reducible for the $k$-regulous topology for any natural integer $k$.
\begin{figure}[ht]
\centering
\includegraphics[height =6cm]{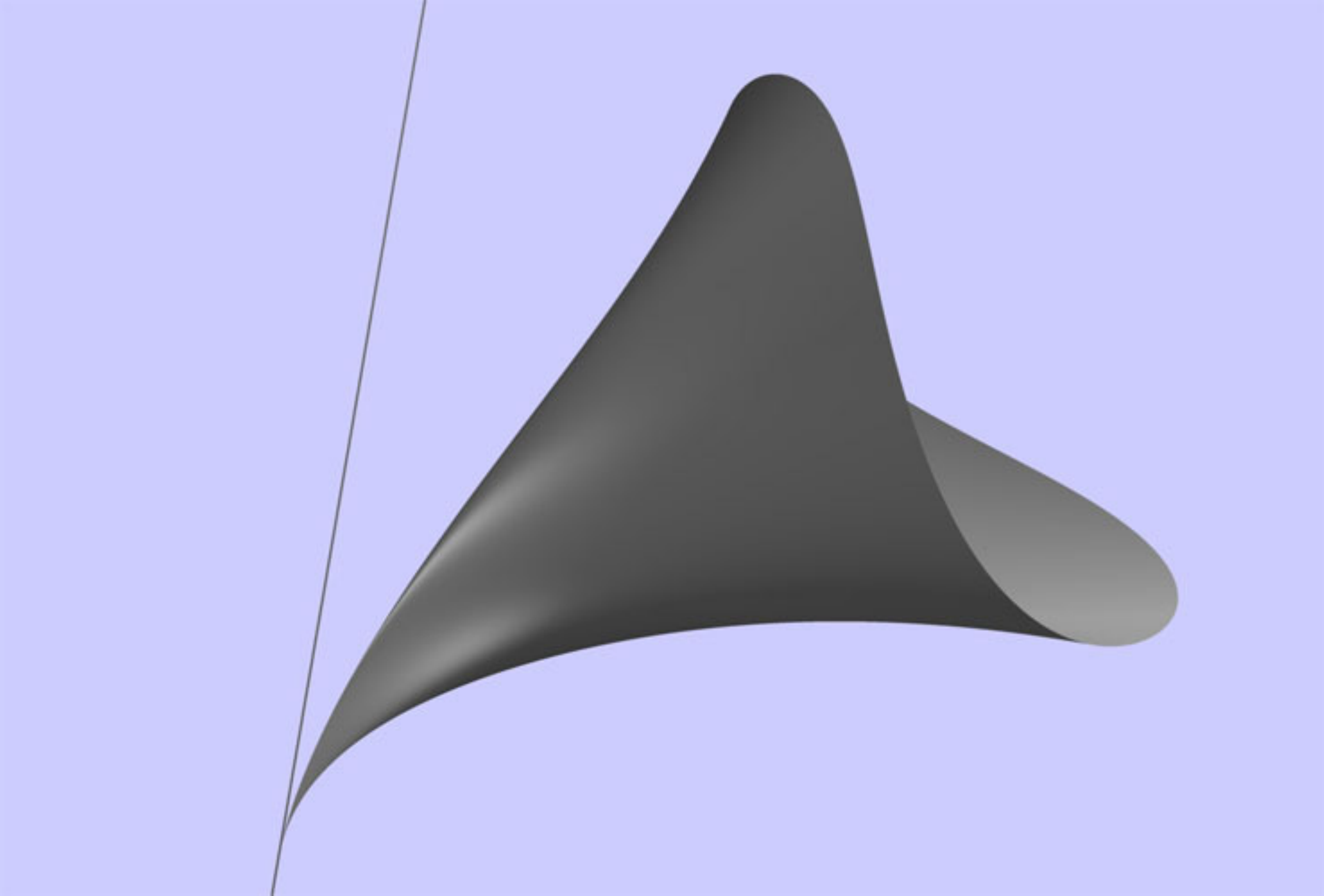}
\caption{A horned umbrella: $x^2+y^2\bigl((y-z^2)^2+yz^3\bigr)=0$.}
        \label{fig.horned}
\end{figure}

In fact, the "horn" of the umbrella is closed for the $0$-regulous topology as it is the zero set of the regulous function 
$$
(x,y,z) \mapsto z^2\frac{x^2+y^2\bigl((y-z^2)^2+yz^3\bigr)}{x^2+y^4+y^2z^4}\;.
$$ 

The "stick" of the umbrella is also closed: it is the zero set of $(x,y,z) \mapsto x^2+y^2$, thus the umbrella is reducible for the regulous topology, see \cite[Exemple~6.12]{regul} for details.

In the paper \cite{regul}, several properties of the rings $\cR^k(\RR^n)$ are established. In particular, a strong Nullstellensatz is proved. The scheme theoretic properties are studied and regulous versions of Theorems A and B of Cartan are proved. There is also a geometrical characterization of prime ideals of $\cR^k(\RR^n)$ in terms of the zero-locus of regulous functions and a relation between $k$-regulous topology and the topology generated by euclidean closed Zariski-constructible sets. Many papers are related to this new line of research and among them we recommend:
\cite{Ku13,BKVV13,KK13,Ku14-JEMS,Ku14-Arch,No14} \cite{FMQ14}.

\backmatter

\bibliographystyle{smfalpha}
\bibliography{biblio-perso,biblio-tvar,biblio-preprints}

\end{document}